\documentclass[12pt]{article}
\usepackage{epsfig}
\usepackage{latexsym}
\usepackage{graphicx}
\usepackage{amsmath}
\usepackage{amsfonts}
\usepackage{amssymb}

\usepackage{color}

\newtheorem{theorem}{Theorem}

\newtheorem{lemma}{Lemma}
\newtheorem{corollary}{Corollary}
\newtheorem{conjecture}{Conjecture}

\newtheorem{claim}{Claim}

\newcommand{\pind}{{\rm pind}}
\newcommand{\per}{{\rm per}}

\newenvironment{proof}{
\par
\noindent {\bf Proof.}\rm}{\mbox{}\hfill\rule{0.5em}{0.809em}\par}

\begin{document}
\title{Permanent index of matrices associated with graphs}
\author{
Tsai-Lien Wong \thanks{Department of Applied Mathematics, National
Sun Yat-sen University, Kaohsiung, Taiwan 80424. Grant numbers: MOST
104-2115-M-110 -001 -MY2. Email: tlwong@math.nsysu.edu.tw} \and
Xuding Zhu
\thanks{Department of Mathematics, Zhejiang Normal University,
China. Grant Numbers: NSF11171310 and ZJNSF  Z6110786. Email:
xudingzhu@gmail.com. }
        \\[0.2cm]
       }

\date{}

\maketitle

\begin{abstract}
A total weighting of a graph $G$ is a mapping $f$ which assigns to
each element $z \in V(G) \cup E(G)$ a real number  $f(z)$ as its
weight. The vertex sum of $v$ with respect to $f$ is
$\phi_f(v)=\sum_{e \in E(v)}f(e)+f(v)$. A total weighting is proper
if $\phi_f(u) \ne \phi_f(v)$ for any edge $uv$ of $G$. A
$(k,k')$-list assignment is a mapping $L$ which assigns to each
vertex $v$ a set $L(v)$ of $k$ permissible weights, and assigns to
each edge $e$ a set $L(e)$ of $k'$ permissible weights. We say $G$
is $(k,k')$-choosable if for any $(k,k')$-list assignment $L$,
there is a proper total weighting $f$ of $G$ with $f(z) \in L(z)$
for each $z \in V(G) \cup E(G)$. It was conjectured in  [T. Wong and
X. Zhu, Total weight choosability of graphs, J. Graph Theory 66
(2011), 198-212] that every graph is $(2,2)$-choosable and every
graph with no isolated edge is $(1,3)$-choosable.   A promising tool
in the study of these conjectures is Combinatorial Nullstellensatz.
This approach leads to conjectures on the permanent indices  of
matrices
 $A_G$ and $B_G$ associated to a graph $G$.   In this
paper, we establish a method that reduces the study of  permanent of
matrices associated to a graph $G$ to the study of permanent of
matrices associated to induced subgraphs of $G$.  Using this
reduction method, we show that if $G$ is a  subcubic graph, or a
$2$-tree, or a  Halin graph, or a grid, then  $A_G$ has permanent
index $1$. As a consequence,  these graphs are
$(2,2)$-choosable.
\end{abstract}
{\small \noindent{{\bf Key words: }  Permanent index, matrix, total
weighting}

\section{Introduction}

A {\em total weighting}  of a graph $G$ is a mapping $f$ which assigns to each element
$z \in V(G) \cup E(G)$ a real number  $f(z)$ as its weight.
Given a total weighting $f$ of $G$, for a vertex $v$ of $G$,  the {\em vertex sum} of $v$ with respect to $f$ is defined as
$\phi_f(v)=\sum_{e \in E(v)}f(e)+f(v)$. A total weighting is proper if $\phi_f$ is a proper colouring of $G$,
i.e., for any edge $uv$ of $G$, $\phi_f(u) \ne \phi_f(v)$.
A total weighting $\phi$ with $\phi(v)=0$ for all vertices $v$ is also called
an {\em edge weighting}.  A  proper edge weighting $\phi$ with $\phi(e) \in \{1,2,\ldots, k\}$ for all edges $e$ is called a  {\em vertex colouring  $k$-edge weighting} of $G$.
Karonski, {\L}uczak and Thomason \cite{KLT2004} first studied edge weighting of graphs.
They conjectured that  every
graph with no isolated edges has a vertex colouring $3$-edge weighting.
This conjecture received considerable
attention, and is called the 1-2-3 conjecture.   Addario-Berry, Dalal, McDiarmid, Reed and
Thomason \cite{Add2007}  proved that every graph with no isolated edges has a vertex colouring $k$-edge weighting for $k=30$. The bound $k$  was   improved to $k
= 16$ by Addario-Berry, Dalal and Reed in \cite{Add2005} and to $k =
13$ by Wang and Yu in \cite{yu}, and to $k=5$ by  Kalkowski
\cite{K}.

Total weighting of graphs was first studied by Przyby{\l}o and  Wo\'{z}niak in \cite{PW2010}, where they defined
$\tau(G)$ to be the least integer $k$ such that $G$ has
  a proper total weighting $\phi$ with $\phi(z) \in \{1,2, \ldots, k\}$ for $z \in V(G) \cup E(G)$.
  They proved that $\tau(G) \le 11$ for all graphs $G$, and  conjectured  that $\tau(G)=2$ for all graphs $G$.
  This conjecture is called the 1-2 conjecture.
A breakthrough on 1-2 conjecture was obtained by Kalkowski,
Karo\'{n}ski and   Pfender  in \cite{KKP10}, where it was proved
that every graph $G$ has a proper total weighting $\phi$ with
$\phi(v) \in \{1,2\}$ for $v \in V(G)$ and $\phi(e) \in \{1,2,3\}$
for $e \in E(G)$.

The list version of edge weighting  of graphs was introduced by
Bartnicki,   Grytczuk and   Niwczyk in \cite{BGN09}, and the list
version of total weighting of graphs was introduced independently by
Wong and Zhu in \cite{WZ11} and by Przyby{\l}o and Wo\'{z}niak
\cite{PW2011}. Suppose $\psi: V(G) \cup E(G) \to \{1,2,\ldots,\}$ is
a mapping which assigns to each vertex and each edge of $G$ a
positive integer. A $\psi$-list assignment of $G$ is a mapping $L$
which assigns to $z \in V(G) \cup E(G)$ a set $L(z)$ of $\psi(z)$
real numbers. Given a total list assignment $L$, a proper $L$-total weighting
is a proper total weighting $\phi$ with $\phi(z) \in L(z)$ for all $z \in V(G) \cup E(G)$.
We say
$G$ is {\em total weight $\psi$-choosable} if for any $\psi$-list
assignment $L$, there is a proper $L$-total weighting  of $G$. We say $G$ is
$(k,k')$-choosable if $G$ is $\psi$-total weight choosable, where
$\psi(v)=k$ for $v \in V(G)$ and $\psi(e) = k'$ for $e \in E(G)$.

As strengthening of the 1-2-3 conjecture and the 1-2 conjecture, it
was conjectured in \cite{WZ11} that  every graph with no isolated
edges is $(1,3)$-choosable and every graph is $(2,2)$-choosable.
Thes two conjectures received a lot of attention and are verified for some special classes of graphs.
In particular, it was shown in \cite{ZW2014} that every graph is $(2,3)$-choosable.
A promising tool in the study of these conjectures is
Combinatorial Nullstellensatz.  For each $z  \in V(G) \cup E(G)$, let $x_z$
be a variable associated to $z$. Fix an  orientation $D$ of $G$.
Consider the
polynomial
$$P_G(\{x_z: z \in V(G) \cup E(G)\}) = \prod_{e=uv \in E(D)}\left(  \left(\sum_{e \in E(v)} x_e+ x_v\right) - \left(\sum_{e \in E(u)} x_e+ x_u\right)\right).$$
Assign a real number $\phi(z)$ to the variable $x_z$, and view
$\phi(z)$ as the weight of $z$.
Let $P_G( \phi  )$ be the evaluation of the
polynomial at $x_z = \phi(z)$. Then $\phi$ is a proper total
weighting of $G$ if and only if $P_G( \phi) \ne 0$.
Note that $P_G$ has degree $|E(G)|$.

An {\em index function} of $G$ is a mapping $\eta$ which
assigns to each vertex or edge $z$ of $G$ a non-negative integer $\eta(z)$ and
an index function $\eta$ is {\em valid} if
$\sum_{z \in V(G) \cup E(G)}\eta(z)=|E(G)|$.
For a valid  index function $\eta$, let
$c_{\eta}$ be the coefficient of the monomial
 $\prod_{z \in V \cup E} x_{z}^{\eta(z)}$ in the expansion of $P_G$.
It follows from  Combinatorial Nullstellensatz
\cite{nullstellensatz,AlonTarsi} that if $c_{\eta} \ne 0$, and
$L$ is a list assignment which assigns to each $z \in V(G) \cup E(G)$ a
set $L(z)$ of  $ \eta(z)+1$ real numbers, then  there exists a
mapping $\phi$ with  $\phi(z) \in L(z)$   such that $$P_G(\phi) \ne
0.$$
So to prove a graph $G$ is $(k,k')$-choosable, it suffices to show that
there is a valid index function $\eta$ with $\eta(v) \le k-1$ for $v \in V(G)$, $\eta(e) \le k'-1$ for $e \in E(G)$  and $c_{\eta} \ne 0$.

We write the polynomial $P_G(\{x_z: z \in V(G) \cup E(G)\})$  as
$$P_G(\{x_z: z \in V(G) \cup E(G)\}) = \prod_{e \in E(D)}\sum_{z \in V(G) \cup E(G)}A_G[e,z]x_z.$$
It is straightforward to verify that for $e \in E(G)$ and $z \in V(G) \cup E(G)$,
if $e=(u,v)$ (oriented from $u$ to $v$), then
\begin{equation*}
A_G[e,z]=
\begin{cases} 1 & \text{if $z=v$, or $z \ne e$ is an edge incident to $v$,}
\\
-1 & \text{if $z=u$, or $z \ne e$ is an edge incident to $u$,}
\\
0 &\text{otherwise.}
\end{cases}
\end{equation*}
Now $A_G $ is a matrix, whose rows  are indexed by the edges of $G$ and the columns are indexed by edges and vertices of $G$. Let $B_G$ be the submatrix of $A_G$ consisting of those columns of $A_G$ indexed by edges. It turns out that $(k,k')$-choosability of a graph $G$ is related to the permanent indices of $A_G$ and $B_G$.

For  an $m \times m$ matrix $A$ (whose entries are reals), the {\em
permanent} of $A$  is defined as
$${\rm per}(A)=\sum_{\sigma \in S_m} \prod_{i=1}^mA[i, \sigma(i)]$$
where $S_m$ is the symmetric group of order $m$, i.e., the summation
is taken over all the permutations $\sigma$ over $\{1,2,\ldots,
m\}$.
The {\em  permanent index}
of a matrix $A$, denoted by $\pind(A)$,  is the minimum integer $k$ such that there is a matrix $A'$ such that $\per(A') \ne 0$, each column of $A'$ is a column of $A$ and each column of $A$ occurs in $A'$ at most $k$ times (if such an integer $k$ does not exist, then $\pind(A)=\infty$).

Consider the matrix $A_G$ defined above.
Given a vertex or edge $z$ of $G$, let $A_G(z)$ be the column of
$A_G$ indexed by $z$. For an index function $\eta$ of $G$, let
$A_G(\eta)$ be the
 matrix, each of
its column is a column of $A_G$, and each column $A_G(z)$ of $A_G$
occurs $\eta(z)$ times as  a column of $A_G(\eta)$. It is known
  \cite{alontarsi1989, WZ11}
and easy to verify that for a valid index function $\eta$ of $G$,
$c_{\eta} \ne 0$ if and only if $\per(A_G(\eta)) \ne 0$. Thus  if
$\pind(A_G) =1$, then  $G$ is $(2,2)$-choosable;
 if $\pind(B_G) \le 2$,
 then $G$ is $(1,3)$-choosable. The
following two conjectures are proposed in \cite{WZ11}:

\begin{conjecture} \cite{BGN09}
\label{g0} For any graph $G$ with no isolated edges,   $\pind(B_G) \le
2$.
\end{conjecture}

\begin{conjecture}\cite{WZ11}
\label{g1} For any graph $G$,  $\pind(A_G) =
1$.
\end{conjecture}

The discussion above shows that Conjecture \ref{g0} implies that any graph without
isolated edges is $(1,3)$-choosable,  and
Conjecture \ref{g1} implies that every graph is $(2,2)$-choosable.

We say an index function $\eta$ is {\em  non-singular}  if there is
a valid index function  $\eta' \le \eta$ with $\per(A_G(\eta')) \ne
0$. In this paper, we are interested in non-singularity of index
functions $\eta$ for which $\eta(e)=1$ for every edge $e$ and $\eta(v)$ can be any non-negative integers for any every
vertex $v$. Assume
$\eta$ is such an index function of $G$. We delete a vertex $v$, and
construct an index function $\eta'$ for $G-v$ from the restriction
of $\eta$ to $G-v$ by doing the following modification: $\eta(v)$ of
the neighbours $u$ of $v$ have $\eta'(u)=\eta(u)+1$, and all the
other neighbours $u$ of $v$ (if any) have $\eta'(u) = \eta(u)-1$. We
prove that if $\eta'$ is a non-singular index function of $G-v$,
then $\eta$ is a non-singular index function of $G$. Applying  this
reduction method,  we prove that Conjecture \ref{g1} holds for
subcubic graphs, $2$-trees,  Halin graphs and grids.
 Consequently,  subcubic graphs,
$2$-trees, Halin graphs and grids  are   $(2, 2)$-choosable.

\section{Reduction to induced subgraphs}
\label{sec:reduction}

To study non-singularity of index functions of $G$,   we shall consider
matrices whose columns are linear combinations of columns of $A_G$.
Assume $A$ is a square matrix whose columns are linear combinations
of columns of $A_G$. Define an index function $\eta_A: V(G) \cup
E(G) \to \{0,1, \ldots, \}$ as follows:

For $z \in V(G) \cup E(G)$,  $\eta_A(z)$ is  the number of columns
of $A$ in which $A_G(z)$ appears  with nonzero coefficient.

It is known \cite{WZ11} that columns of $A_G$ are not linearly independent. In particular,
if $e=uv$ is an edge of $G$, then $$A_G(e) = A_G(u)+A_G(v) \eqno(1)$$
Thus a column of $A$ may have different ways to be expressed as
linear combinations of columns of $A_G$. So the index function
$\eta_A$ is not uniquely determined by $A$. Instead, it is
determined by the way we choose to express the columns of $A$ as
linear combinations of columns of $A_G$. For simplicity, we use the
notation $\eta_A$, however, whenever the function $\eta_A$ is used,
an explicit expression of the columns of $A$ as linear combinations
of columns of $A_G$ is given, and we refer to that specific
expression.

It is well-known (and follows easily from the definition) that
the permanent of a matrix is multi-linear on its column vectors and row vectors:
If a column   $C$ of  $A$ is
a linear combination of two columns vectors $C =\alpha C'+ \beta C''$,
 and $A'$ (respectively, $A''$) is obtained from $A$ by replacing the
column  $C$ with $C'$ (respectively,  with
$C''$), then
$${\rm per}(A) = \alpha\ {\rm per}(A') + \beta \ {\rm per}(A''). \eqno(2)$$
By using (2) repeatedly, one can find matrices $A_1, A_2, \ldots,
A_q$ and real numbers $a_1, a_2, \ldots, a_q$ such that
$$\per(A)  = \sum_{j=1}^q a_j \per(A_j)$$
where each $A_j$ is a square matrix consisting of columns of $A_G$,
with each column $A_G(z)$ appears at most $\eta(z)$ times. Thus if
$\per(A) \ne 0$, then one of the $\per(A_j) \ne 0$. Thus if $\per(A) \ne 0$,
then $\eta_A$ is a non-singular index function of $G$.

\begin{theorem}
\label{del} Suppose $G$ is a graph, $\eta$ is an index function of $G$ for which $\eta(e)=1$ for every edge $e$.
Let $v$ be a vertex of $G$.
Let $\eta'$ be obtained from the restriction of $\eta$ to $G-v$ by the following modification:
Choose $d_G(v)-\eta(v)$ neighbours $u$ of $v$ with $\eta(u) \ge 1$, and let $\eta'(u)=\eta(u)-1$.
For the other $\eta(v)$ neighbours $u$ of $v$, let $\eta'(u)=\eta(u)+1$.
If $\eta'$ is a non-singular index function of $G-v$, then $\eta$ is a non-singular index function of $G$.
\end{theorem}

Theorem \ref{del} follows from the following more general  statement.

\begin{theorem}
\label{del2} Suppose $G$ is a graph, $v$ is a vertex of $G$ and
$E(v)=\{e_1, e_2,\ldots, e_k\}$, with $e_i=vv_i$ for $i=1,2,\ldots,k$.
Assume $\eta$ is an index function of $G$. Here $\eta(e)$ can be any non-negative integer.
Choose a subset $J$  of $\{1,2,\ldots, k \}$ and integers
$1 \le k_i \le \min\{\eta(e_i), \eta(v_i)\}$  such that $\eta(v) + \sum_{i \in J}
k_i = k$.
Let $\eta'$ be the index function   of $G' = G- v$ which is equal to the restriction
of $\eta$ to $G-v$, except that
\begin{enumerate}
\item For $i \in J$,   $\eta'(v_i) =  \eta(v_i) - k_i$.
\item For $i \in \{1,2,\ldots, k\} \setminus J$,   $\eta'(v_i) = \eta(v_i)+ \eta(e_i)$.
\end{enumerate}
If $\eta'$ is a non-singular index function for $G'$, then $\eta$ is
a non-singular index function for $G$.
\end{theorem}
\begin{proof}
Assume $\eta'$ is non-singular. Let $\eta'' \le  \eta'$ be a valid
index function with $\per(A_{G'}(\eta'')) \ne 0$.

Assume $|E(G)|=m$ and $|E(G')|=m' = m -k$. By viewing each vertex
and each edge of $G'$ as a vertex and an edge of $G$, $A_G(\eta'')$
is an $m \times m'$ matrix, consisting $m'$ columns of $A_G$. First
we extend $A_G(\eta'')$ into an $m \times m$ matrix $A$ by adding $k$
copies of the column $A_G(v)$. The added $k$ columns has $k$ rows
(the rows indexed by edges incident to $v$) that are all $1$'s
(with all these edges oriented towards $v$), and
all the other entries of these $k$ columns are $0$. Therefore
$\per(M) = \per(A_{G'}(\eta'')) k !$, and hence $\per(M) \ne 0$.

Starting from the matrix $M$, for each $i\in \{1,2,\ldots, k\}
\setminus J$, remove $\min\{\eta(e_i), \eta''(v_i)\}$ copies of the column $A_G(v_i)$ and
add $\min\{\eta(e_i), \eta''(v_i)\}$ copies of the column $A_G(e_i)$. Denote by $M'$ the
resulting matrix.

\begin{claim}
For the matrix $M'$ constructed above, we have $\per(M') = \per(M)$.
\end{claim}
\begin{proof}
Since by (1),
 $A_G(e_i) = A_G(v_i) + A_G(v)$, we re-write
$\min\{\eta(e_i), \eta''(v_i)\}$ copies of the column $A_G(e_i)$ of
$M'$ as  $A_G(v) + A_G(v_i)$. Then we expand the permanent using its
multilinear property (i.e. using (2) repeatedly), to obtain the
following equation:
$$\per(M')=\per(M)+ \sum_{M''} \per(M'')$$
where $M''$ are those matrices which contain at least $k+1$ copies
of the column $A_G(v)$. Since these $k+1$ columns has all $1$'s in
$k$ rows and $0$ in all other entries, we have $\per(M'') = 0$ for
all $M''$, and so $\per(M') = \per(M)$.
\end{proof}

For each $i \in J$, write $k_i$ copies of $A_G(v)$ in $M'$ as
$A_G(e_i)-A_G(v_i)$. Note that this step does not change the matrix,
since $A_G(v) = A_G(e_i)-A_G(v_i)$
(by (1)). Now each column of $M'$
is a linear combination of columns of $A_G$.

 We  shall show that, with the linear combination of columns of $M'$ given in the above paragraph,
$\eta_{M'}(z) \le \eta(z)$ for   $z \in V(G) \cup E(G)$.

If $z \notin \{e_i, v_i: i=1,2,\ldots, k\} \cup \{v\}$,
$\eta_{M'}(z) = \eta_{M}(z) \le \eta''(z) \le \eta'(z) = \eta(z)$.
If $i \in \{1,2,\ldots, k\} - J$, then $\eta_{M'}(e_i) = \min
\{\eta(e_i), \eta''(v_i)\} \le \eta(e_i)$, and  $
\eta_{M'}(v_i)=\eta_{M}(v_i) - \min \{\eta(e_i), \eta''(v_i)\} \le
\max\{0,\eta''(v_i)-\eta(e_i)\} \le \eta'(v_i)-\eta(e_i)=\eta(v_i)$.
If $i \in J$, then $\eta_{M'}(e_i) = k_i \leq \eta(e_i)$ and
$\eta_{M'}(v_i)=\eta''(v_i)+k_i \le \eta'(v_i)+k_i = \eta(v_i)$.
Finally, $\eta_{M'}(v)=k - \sum_{i \in J}k_i = \eta(v)$. As
$\per(M') \ne 0$, we conclude that $\eta$ is a non-singular index
function for $G$. This completes the proof of Theorem \ref{del2}.
\end{proof}

Theorem \ref{del} follows from Theorem \ref{del2} by choosing $k_i=1$  and $|J|=d(v)-\eta(v)$.
By definition, if $\eta''$ is non-singular and $\eta' \ge \eta''$, then $\eta'$ is also non-singular. So the following is
equivalent to Theorem \ref{del}.
\begin{theorem}
\label{del3} Suppose $G$ is a graph, $\eta$ is an index function of $G$ for which $\eta(e)=1$ for every edge $e$.
Let $v$ be a vertex of $G$.
Let $\eta'$ be obtained from the restriction of $\eta$ to $G-v$ by the following modification:
Choose at least $d_G(v)-\eta(v)$ neighbours $u$ of $v$ with $\eta(u) \ge 1$, and let $\eta'(u)=\eta(u)-1$.
For the other  neighbours $u$ of $v$, let $\eta'(u)=\eta(u)+1$.
If $\eta'$ is a non-singular index function of $G-v$, then $\eta$ is a non-singular index function of $G$.
\end{theorem}

We shall apply Theorem \ref{del3} repeatedly and delete a sequence of vertices in order.
We need to record which vertices are deleted, and when a vertex is deleted, for which
neighbours $u$ we have $\eta'(u)=\eta(u)+1$.
For this purpose, instead of really removing the deleted vertices,
we indicate the deletion of $v$ by orient all the edges
incident to $v$ from $v$ to its neighbours, and then choose a subset of these oriented edges (to indicate those neighbours $u$
for which $\eta'(u)=\eta(u)+1$).

The index function $\eta$ is changing in the process of the deletion.
For convenience, we denote by $\eta_i$ the index function after the deletion of the
$i$th vertex. In particular, $\eta_0=\eta$.

Assume a vertex $v$ is deleted in the $i$th step, for each neighbour
$u$ of $v$ (at the time $v$ is deleted), orient the edge as an arc
from $v$ to $u$. After a sequence of vertices are deleted, we obtain
a digraph $D$  formed by edges incident to the ``deleted" vertices.
Let $D'$ be the sub-digraph  of¡@ $D$ formed  by those arcs $(v, u)$
with $u$ be the neighbour of $v$ (at the time $v$ is deleted) and
for which we have  $\eta'(u) = \eta(u)+1$.

If $u$ is deleted in the $i$th step, then
 $d_{D'}^+(u) \le \eta_{i-1}(u)$. After the $i$th step, all edges incident to $u$ are oriented.
On the other hand, $d_{D'}^-(u)$ is  the number of indices $j <i$ for which $\eta_j(u)=\eta_{j-1}(u)+1$,
and  $d_D^-(u)-d_{D'}^-(u)$ is   the number of indices $j < i$  for which $\eta_j(u)=\eta_{j-1}(u)-1$.
Thus $d_{D'}^+(u) \le \eta(u)+d_{D'}^-(u)-(d_D^-(u)-d_{D'}^-(u))$.

If after the $i$th step, $u$ is not deleted, then $d_{D'}^+(u)=0$ and $\eta_i(u)=\eta(u)+d_{D'}^-(u)-(d_D^-(u)-d_{D'}^-(u)) \ge 0$.

The following corollary summarize the final effect of the repeated application of Theorem \ref{del}.

\begin{corollary}
\label{delete} Suppose $G$ is a graph, $\eta$ is  an index
function of $G$ with $\eta(e)=1$ for all edges $e$,  and $X$ is a subset of $V(G)$.
Let $G'=G-E[X]$ be obtained from $G$ by deleting
edges in $G[X]$. Let   $D$  be an acyclic orientation   of $G'$, in which
each vertex $v \in X$ is a sink.
Assume $D'$ is a sub-digraph of $D$ such that for all $v \in V(D)$,
$$\eta(v)+2d^-_{D'}(v) - d^-_D(v) \ge d^+_{D'}(v), \eqno(*)$$
 Let $\eta'$ be the index function defined as $\eta'(e)=1$ for every edge $e$ of $G[X]$ and
 $\eta'(v)=\eta(v)+ 2d_{D'}^-(v) - d_D^-(v)$
 for $v \in X$. If  $\eta'$ is a non-singular index function  for
$G[X]$, then $\eta$ is a non-singular index function for $G$.
\end{corollary}
\begin{proof}
Assume $\eta'$ is non-singular for $G[X]$. We shall prove that $\eta$ is non-singular for $G$.
We prove this by induction on $|V-X|$. If $V-X = \emptyset$, then
$\eta =\eta'$ and there is nothing to prove.

Assume $V-X \ne \emptyset$. Since the orientation $D$ is acyclic,
there is a source vertex $v \notin X$. Let $e_1, e_2, \ldots, e_k$
be the set of edges incident to $v$ and $e_i=vv_i$.

Consider the graph $G-v$. Let $\eta''$ be the index function on
$G-v$ defined as $\eta'' =  \eta$ on $G-v$, except that for
$i=1,2,\ldots, k$, if $e_i \notin D'$,   then  $\eta''(v_i) =
\eta(v_i)-1$, and if $e_i \in D'$, then  $\eta''(v_i) = \eta(v_i)+
1$.

Let $H=D-v$ and $H'=D'-v$. We shall show that
$$\eta''(u)+ 2d^-_{H'}(u) - d^-_H(u) \ge d^+_{H'}(u) {\mbox { for all $u\in V(H)$}} \eqno(**)$$
If $u \notin \{v_1, v_2, \ldots, v_k\}$, then ($**$) is the same as ($*$).
If $u=v_i$ and $e_i\in D'$, then $\eta''(v_i)=\eta(v_i)+1, d_{H'}^-(v_i)=d_{D'}^-(v_i)-1, d_H^-(v_i)=d_D^-(v_i)-1$
and $d_H^+(v_i)=d_D^+(v_i)$. So ($**$) follows from ($*$).
If $u=v_i$ and $e_i\notin D'$, then $\eta''(v_i)=\eta(v_i)-1, d_{H'}^-(v_i)=d_{D'}^-(v_i), d_H^-(v_i)=d_D^-(v_i)-1$
and $d_H^+(v_i)=d_D^+(v_i)$. Again  ($**$) follows from ($*$).

Therefore, by induction hypothesis,  $\eta''$ is non-singular for $G-v$.
Apply Theorem \ref{del} to $\eta''$ and $\eta$,   with
  $J=\{i: 1 \le i \le k, e_i \notin
D'\}$ and $k_i=1$ for $i \in J$, we conclude  that $\eta$ is non-singular
for $G$.
\end{proof}

\section{Application of the reduction method}

\begin{lemma}
\label{subcubic} Suppose $G$ is a connected graph, and
$\eta$ is an index function with $\eta(e) =1 $ for all $e \in E(G)$.  Assume one of the following holds:
\begin{itemize}
\item  $\eta(v)\ge \max \{1, d_G(v)-2\}$ for every vertex $v$.
\item Each vertex $v$ has  $\eta(v) \ge d_G(v)-2$ and at least one vertex $v$ has $\eta(v)\ge d_G(v)$.
\end{itemize}
Then $\eta$ is a non-singular index function of $G$.
\end{lemma}
\begin{proof}
Assume the lemma is not true and $G$ is a counterexample with
minimum number of vertices.

Assume first that $\eta(v) \ge \max \{1, d_G(v)-2\}$ for all $v$.
By reducing the value of $\eta$ if needed, we may assume that
 $\eta(v) = \max \{1, d_G(v)-2\}$ .
Let $v$ be a non-cut vertex of $G$ and let $v_1, \ldots, v_k$   be
the neighbours of $v$. Consider the graph $G-v$. Let $\eta'$ be the
index function of $G-v$ defined as $\eta'=\eta$, except that
$\eta'(v_i)=\eta(v_i)-1$ for $i=1,2,\ldots, k-1$ and
$\eta'(v_k)=\eta(v_k)+1$. For each $i \in \{1,2,\ldots, k-1\}$, we
have $\eta'(v_i) \ge d_{G-v}(v_i)-2$, and $\eta'(v_k) \ge
d_{G-v}(v_k)$. As $G-v$ is connected, the condition of the lemma is
satisfied by $G-v$ and $\eta'$. By the minimality of $G$,  $\eta'$
is a non-singular index function for $G-v$. By Theorem \ref{del},
 $\eta$ is a
non-singular index function for $G$.

Assume each vertex $u$ has  $\eta(u) \ge d_G(u)-2$ and  one vertex $v$ has $\eta(v)\ge d_G(v)$.
Let $\eta'$ be the index function of $G-v$ defined as
$\eta'=\eta$ except that $\eta'(u) = \eta(u)+1$  for all neighbours $u$ of $v$.
Note that for all the neighbours $u$ of $v$, $\eta'(u) \ge d_{G-v}(u)$.
Thus each component of $G-v$, together with  $\eta'$, satisfies the condition of the lemma. By the minimality of $G$,
  $\eta'$ is a non-singular index function for $G-v$.
Apply Theorem \ref{del} again, we conclude that $\eta$ is a
non-singular index function for $G$.
\end{proof}

A graph $G$ is called {\em subcubic} if $G$ has maximum degree at
most $3$.

\begin{corollary}
\label{cubic}
Conjecture \ref{g1} holds for subcubic graphs, i.e.,
if $G$ is a subcubic graph, then $\pind(A_G)=1$.
As a consequence, subcubic graphs are $(2,2)$-choosable.
\end{corollary}
\begin{proof}
If $G$ has maximum degree at most $3$, then it follows from Lemma \ref{subcubic} that
$\eta(z)=1$ for all $z \in V(G) \cup E(G)$ is a non-singular index function.
\end{proof}

A graph $G$ is a {\em $2$-tree} if there is an acyclic orientation   of $G$ (also denoted by $G$)
such that the following hold:
(1) there are two adjacent vertices  $v_0, v_1$ with $d_G^+(v_i)=i$ ($i=0,1$).
(2)   every other vertex   $v$ has $d_G^+(v)=2$,
and  the two out-neighbours of $v$ are adjacent.
If  $N_G^+(v)=\{u,w\}$ and $(u,w)$ is an arc, then $v$ is called a {\em son} of the arc $e=(u,w)$.
For an acyclic oriented graph $G$, for $v \in V(G)$, let $\rho_G(v)$ be the length of the longest directed path
ending at $v$. So if $v$ is a source, then $\rho_G(v)=0$.

\begin{theorem}
\label{2-tree}
Let $G$ be a $2$-tree   and let
 $\eta$ be an index function of $G$. Assume   $\eta(z) \ge 1$ for all $z \in E(G) \cup V(G)$,
 except that possibly there is one  arc $(u,w)$ with $\rho_G(u) \le 1$,
 for which $\eta(w) \ge 0$ and  $\eta(u) \ge 2$.
Then $\eta$ is non-singular for $G$.
\end{theorem}
\begin{proof}
Assume the theorem is not true and $G$ is a counterexample with
minimum number of vertices. If the special arc $(u,w)$ specified in
the theorem does not exist, then let $e=(u,w)$ be an arc which has
at least one son, and with $\rho_G(u)=1$. Note that all the sons of
$e$ are sources. Let $v$ be a son of $(u,w)$ and let $\eta'$ be the
index function of $G'=G-v$ which is equal to $\eta$, except that
$\eta'(u)=\eta(u)+1 \ge 2$ and $\eta'(w)=\eta(w)-1 \ge 0$. Then $G'$
and $\eta'$  satisfying the condition of the theorem, with $e$ be
the special edge (note that $\rho_{G-v}(u) \le \rho_G(u) = 1$).
Hence $\eta'$ is non-singular for $G'$. It follows from Theorem
\ref{del} that $\eta$ is non-singular for $G$.

Assume the special arc  $e=(u,w)$ exists. If $u$ is a source, then
delete $u$, and let $\eta'$ be the index function of $G'=G-u$ which
is equal to $\eta$, except that    $\eta'(v)=\eta(v)+1 $ for
neighbours $v$ of $u$. Then $\eta'(v) \ge 1$ for each vertex of $G'$, hence
$G'$ and $\eta'$  satisfying the
condition of the theorem. So $\eta'$ is non-singular for $G'$, and
it follows from Theorem \ref{del} that $\eta$ is non-singular for
$G$.

If $u$ is not a source vertex and $e$ has  a son $v$, then $v$ is a
source vertex. We  delete $v$ and let $\eta'$ be the index function
of $G'=G-v$ which is equal to $\eta$, except that
$\eta'(u)=\eta(u)-1$ and $\eta'(w) = \eta(w)+1$. Then $G'$ and
$\eta'$  satisfying the condition of the theorem, and hence $\eta'$
is non-singular for $G'$. It follows from Theorem \ref{del} that
$\eta$ is non-singular for $G$.

If $u$ is not a source vertex and $e$ has no son, then there is an
arc $e'=(u,w')$  which has
a son $a$. Since $\rho_G(u) \le 1$, all the sons of $e'$ are
sources. If $e'$ has more than one son, say $a,b$ are both sons of
$e'$, then let $\eta'$ be the restriction of $\eta$ to $G-\{a,b\}$.
By the minimality of $G$, $\eta'$ is non-singular for $G-\{a,b\}$.
By Corollary \ref{delete} (with $D$ consists of the four arcs
incident to $a,b$ and $D'$ consists of arcs $au, bw'$),   $\eta$ is
non-singular for $G$. Assume $e'$ has only one son $a$. Let $\eta'$
be the restriction of $\eta$ to $G-\{a,u\}$, except that
$\eta'(w)=1$. By the minimality of $G$, $\eta'$ is non-singular for
$G-\{a,u\}$. By Corollary \ref{delete} (with $D$ consists of the
four arcs incident to $a,u$ and $D'$ consists of arcs $aw', uw$),
$\eta$ is non-singular for $G$.
\end{proof}

\begin{corollary}
Conjecture \ref{g1} holds for $2$-trees, i.e., if $G$ is a $2$-tree, then $\pind(A_G)=1$, and hence is $(2,2)$-choosable.
\end{corollary}

\begin{theorem}
\label{halin}
If  $T$ is a tree with leaves $v_1, v_2, \ldots, v_n$, and $G$ is obtained from $T$ by
adding edges $v_iv_{i+1}$ ($i=1,2,\ldots, n$, with $v_{n+1}=v_1$), then $\pind(A_G)=1$, and hence $G$ is $(2,2)$-choosable.
\end{theorem}
\begin{proof}
First we construct an acyclic orientation of $G$ as follows: We
choose a non-leaf vertex $u$ of $T$ as the root of $T$. Orient the
edges of the tree from father to son. Then orient the added edges
from $v_i$ to $v_{i+1}$ for $i=1,2,\ldots, n-1$, and orient the edge
$v_1v_n$ from $v_1$ to $v_n$. The resulting digraph is $D$. Now we
choose a sub-digraph $D'$ of $D$ as follows: $D'$ consists of a
directed path $P$ from the root vertex $u$ to $v_1$, and all the
edges $v_iv_{i+1}$ for $i=1,2,\ldots, n-1$, and the edge $v_1v_n$.
Let $\eta$ be the constant function $\eta \equiv 1$, let $X =
\{v_n\}$ and let $\eta'(v_n)=0$, which is an index function of
$G[X]$. Then $\eta'$ is a non-singular index function
of $G[X]$. To prove that $\pind(A_G)=1$, i.e., $\eta$ is a
non-singular index function of $G$, it suffices,
 by Corollary \ref{delete},
to show that for each vertex $v$,
$$ 1 +2d^-_{D'}(v) - d^-_D(v) \ge d^+_{D'}(v).$$

 This is a routine check.  Assume first that $v$ is not a leaf of $T$.
 \begin{enumerate}
 \item If $v$ is not on   path $P$, then $d_{D'}^-(v)= 0$, $d^-_D(v) = 1$ and $d_{D'}^+(v)= 0$. So $ 1
+2d^-_{D'}(v) - d^-_D(v)= 0 \geq d_{D'}^+(v)$.
 \item If $v$ is on $P$, but is not the root $u$, then $d_{D'}^-(v)= 1$, $d^-_D(v) =1 $ and $d_{D'}^+(v)=1 $. So $ 1
+2d^-_{D'}(v) - d^-_D(v)= 2 \geq d_{D'}^+(v)$.
\item If $v=u$, then $d_{D'}^-(v)=0 $, $d^-_D(v) =0 $ and $d_{D'}^+(v)=1 $. So $ 1
+2d^-_{D'}(v) - d^-_D(v)= 1 \geq d_{D'}^+(v)$.
\end{enumerate}
Next, consider the case that $v$ is a leaf of $T$.
\begin{enumerate}
 \item If $v=v_1$, then $d_{D'}^-(v)= 1$, $d^-_D(v) = 1$ and $d_{D'}^+(v)= 2$. So $ 1
+2d^-_{D'}(v) - d^-_D(v)= 2 \geq d_{D'}^+(v)$.
 \item If $v=v_i$, for  $1< i< n$, then $d_{D'}^-(v)=1 $, $d^-_D(v) = 2$ and $d_{D'}^+(v)= 1$. So $ 1
+2d^-_{D'}(v) - d^-_D(v)= 1 \geq d_{D'}^+(v)$.
 \item If $v=v_n$, then $d_{D'}^-(v)= 2$, $d^-_D(v) = 3$ and $d_{D'}^+(v)= 0$. So $ 1
+2d^-_{D'}(v) - d^-_D(v)= 2 \geq d_{D'}^+(v)$.
\end{enumerate}
\end{proof}

A {\em Halin graph} is a planar graph obtained by taking a plane tree (an embedding of a tree on the plane) without degree $2$ vertices
by adding a cycle connecting the leaves of the tree cyclically.

\begin{corollary}
\label{Ha}
Conjecture \ref{g1} holds for Halin graphs, i.e., if $G$ is a Halin graph, then $\pind(A_G)=1$, and hence is $(2,2)$-choosable.
\end{corollary}

A {\em grid} is the Cartesian product of two paths, $P_n \Box P_m$, with vertex set $$V=\{(i,j): 1 \le i \le n, 1 \le j \le m\}$$
and edge set $$E=\{(i,j)(i', j'): i=i',j'=j+1 \ {\rm or} \ i'=i+1, j'=j\}.$$

\begin{lemma}
\label{grids} Assume $m,n \ge 1$. Let  $\eta$ be an index function
of $P_n \Box P_m$, with $\eta(e)=1$ for edges $e$, and one of the
following holds:
\begin{enumerate}
\item[1] $\eta(v)=1$ for all vertices $v$.
\item[2] $\eta(v)=1$ for all vertices $v$, except that  $\eta(n,1)=0$,
and $\eta((n,j))=2$ for $2 \le j \le m$.
\end{enumerate}
Then $\eta$ is non-singular for $G$.
\end{lemma}
\begin{proof} We prove it by induction on the number of vertices of $G$. The
case   $n=1$ or $m=1$ is easy and omitted. Assume $n,m \ge 2$. If
$\eta(v)=1$ for all vertices $v$, then we
delete vertices $(n,1), (n,2), \ldots, (n,m)$ in this order. When
deleting $(n,1)$, we increase $\eta(n,2)$ by $1$ and decrease
$\eta(n-1,1)$ by $1$. When deleting $(n,j)$ for $j \ge 2$, we
increase $\eta(n,j+1)$ by $1$ and increase $\eta(n-1,j)$ by $1$.
After all the vertices  $(n,1), (n,2), \ldots, (n,m)$ are deleted, we obtain a grid
$P_{n-1} \Box P_m$ and an index function $\eta'$  which satisfies the condition of the lemma and hence
is non-singular.  By Theorem \ref{del}, $\eta$ is non-singular.

Assume $\eta(n,1)=0$ and
$\eta(n,j)=2$ for $2 \le j \le m$.  We delete vertices $(n,m),
(n,m-1), \ldots, (n,1)$ in this order, and need not to change $\eta$
except for while deleting $(n, 2)$, we increase $\eta(n,1)$ by $1$.
It follows from induction hypothesis that the resulting index
function is non-singular for $P_{n-1} \Box P_{m}$, and by Theorem
\ref{del} that the original index function $\eta$ is non-singular
for $G$.
\end{proof}

\begin{corollary}
\label{Grids} Conjecture \ref{g1} holds for grids, and hence grids
are $(2,2)$-choosable.
\end{corollary}

\end{document}